\pdfoutput=1
\RequirePackage{ifpdf}
\ifpdf 
\documentclass[pdftex]{sigma}
\else
\documentclass{sigma}
\fi

\usepackage{mathrsfs}

\newtheorem{Theorem}{Theorem}[section]
\newtheorem{Corollary}[Theorem]{Corollary}
\newtheorem{Lemma}[Theorem]{Lemma}
\newtheorem{Proposition}[Theorem]{Proposition}
 { \theoremstyle{definition}
\newtheorem{Definition}[Theorem]{Definition}

\newtheorem{Remark}[Theorem]{Remark} }

\newcommand{\Aa}{\mathcal{A}}
\newcommand{\C}{\mathbb{C}}
\newcommand{\Complex}[1][]{(d_{#1}, \Hh_k)}
\newcommand{\Core}{C}
\newcommand{\Dd}{\mathscr{D}}
\newcommand{\Hh}{\mathscr{H}}
\newcommand{\inj}{\hookrightarrow}
\newcommand{\Jj}{\mathcal{J}}
\newcommand{\LieG}{\mathfrak{g}}
\newcommand{\N}{\mathbb{N}}
\newcommand{\omegat}{\underline{\omega}}
\newcommand{\R}{\mathbb{R}}
\newcommand{\Rr}{\mathcal{R}}
\newcommand{\RepV}{\mathbb{V}}
\newcommand{\RightMult}{\mathcal{R}}
\newcommand{\SymmIdeal}{\mathscr{L}}
\newcommand{\vt}{\underline{v}}

\DeclareMathOperator{\dom}{Dom}
\DeclareMathOperator{\GNS}{GNS}
\DeclareMathOperator{\id}{Id}
\DeclareMathOperator{\ind}{Index}
\DeclareMathOperator{\Ran}{Ran}
\DeclareMathOperator{\Tr}{Tr}

\numberwithin{equation}{section}

\begin{document}

\allowdisplaybreaks

\newcommand{\arXivNumber}{1506.07913}

\renewcommand{\PaperNumber}{016}

\FirstPageHeading

\ShortArticleName{Conformally Twisted Spectral Triples}

\ArticleName{On the Chern--Gauss--Bonnet Theorem\\ and Conformally Twisted
Spectral Triples\\ for $\boldsymbol{C^*}$-Dynamical Systems}

\Author{Farzad FATHIZADEH~$^\dag$ and Olivier GABRIEL~$^\ddag$}

\AuthorNameForHeading{F.~Fathizadeh and O.~Gabriel}

\Address{$^\dag$~Department of Mathematics, Mail Code 253-37, California Institute of Technology,\\
\hphantom{$^\dag$}~1200 E.~California Blvd., Pasadena, CA 91125, USA}
\EmailD{\href{mailto:farzadf@caltech.edu}{farzadf@caltech.edu}}

\Address{$^\ddag$~University of Copenhagen, Universitetsparken 5, 2100 K{\o}benhavn {\O}, Denmark}
\EmailD{\href{mailto:olivier.gabriel.geom@gmail.com}{olivier.gabriel.geom@gmail.com}}
\URLaddressD{\url{http://oliviergabriel.eu/}}

\ArticleDates{Received October 26, 2015, in f\/inal form February 04, 2016; Published online February 10, 2016}

\Abstract{The analog of the Chern--Gauss--Bonnet theorem is
studied for a $C^*$-dynamical system consisting of a~$C^*$-algebra~$A$ equipped with an ergodic action of a~compact Lie group~$G$. The structure of the Lie algebra~$\mathfrak{g}$ of~$G$ is used to interpret the Chevalley--Eilenberg
complex with coef\/f\/icients in the smooth subalgebra $\mathcal{A} \subset A$ as \emph{noncommutative}
dif\/ferential forms on the dynamical system. We conformally perturb
the standard metric, which is associated with the unique
$G$-invariant state on~$A$, by means of a Weyl conformal
factor given by a positive invertible element of the algebra, and consider
the Hermitian structure that it induces on the complex.
A~Hodge decomposition theorem is proved, which allows
us to relate the Euler characteristic of the complex to the
index properties of a~Hodge--de~Rham operator for the
perturbed metric. This operator, which is shown to be selfadjoint, is a~key ingredient
in our construction of a spectral triple on~$\mathcal{A}$ and a twisted
spectral triple on its opposite algebra. The conformal invariance
of the Euler characteristic is interpreted as an indication of the Chern--Gauss--Bonnet theorem
in this setting. The spectral triples encoding the conformally perturbed
metrics are shown to enjoy the same spectral summability
properties as the unperturbed case.}


\Classification{58B34; 47B25; 46L05}

\vspace{-2mm}

\section{Introduction}

In noncommutative geometry \cite{IHES, NCG}, $C^*$-dynamical systems
$(A, G, \alpha)$ have been long studied from a dif\/ferentiable point of
view starting with extending the basic notions of dif\/ferential
geo\-met\-ry and dif\/ferential topology to a dif\/ferential
structure on a $C^*$-algebra $A$ endowed
with an action $\alpha\colon G \to \operatorname{Aut}(A)$ of
a Lie group~$G$. That is, the notion of a~connection, a~vector bundle, and Chern classes were introduced for such a~dynamical system, a pseudodif\/ferential calculus was developed and the analog
of the Atiyah--Singer index theorem was proved in~\cite{CAlgGeoDiff}.
The noncommutative two torus $\mathbb{T}_\theta^2$ has been
one of the main motivating examples for these developments.
In~\cite{TrSpLieGpWahl}, this line of investigation has been
taken further focusing on general compact Lie groups and index theory.

Following the seminal work of Connes and Tretkof\/f on the
Gauss--Bonnet theorem for~$\mathbb{T}_\theta^2$~\cite{GaussBonnet-ConnesTretkoff} and its extension
in \cite{GaussBonnet-FK} concerning general translation invariant conformal
structures, local dif\/ferential geometry of non-f\/lat noncommutative tori
has been a subject of increasing interest
in recent years \cite{RicciBM, ModCurvCM, ScalarCurvNCtorusFK,
DixmierTraceNCtorusFK, ScalarCurv4NCTFK, MoritaLM}. A Weyl conformal factor
may be used to perturb a f\/lat metric on noncommutative tori, and Connes' pseudodif\/ferential calculus
\cite{CAlgGeoDiff} can be employed along with noncommutative computational
methods to carry out calculation of scalar curvature and to investigate
the related dif\/ferential geometric statements, see also \cite{AsymmetricDS}
for an asymmetric perturbation of the metric.

The idea and the techniques were indeed initiated in a preprint \cite{PreprintCC},
where with the help of complicated modif\/ied logarithmic
functions and a modular automorphism, an expression for the
value $\zeta(0)$ of the spectral zeta function of the Laplacian
of a curved metric on $\mathbb{T}^2_\theta$ was written. The vanishing of this expression is
interpreted as the Gauss--Bonnet theorem \cite{GaussBonnet-ConnesTretkoff},
which was suggested by the developments in the following intimately
related theories. In fact, the spectral action principle \cite{SpActChC},
in particular the related calculations in the presence of a dilaton
\cite{SpActScaleChC}, and the theory of twisted spectral triples,
which arise naturally in noncommutative conformal geo\-met\-ry~\cite{TrSpTypeIII, NCConformalPW}, indicate independence of~$\zeta(0)$
from the conformal factor.

Connes' index formula for Fredholm modules, which involves cyclic cohomology,
is quite broad~\cite{IHES}. It asserts that given a f\/initely summable Fredholm module over
an algebra, the analytic index, given by pairing a
$K$-homology and a $K$-theory element of the
algebra, coincides with the topological index, which pairs
the corresponding elements in periodic cyclic cohomology and homology obtained
by the Chern--Connes characters. The local index formula of Connes and Moscovici \cite{localtrace95}
gives a local formula based on residue trace functionals, which is in the same cyclic cohomology
class as the Chern--Connes character, and has the advantage that one can perform
explicit computations with it (see also \cite{localHig}). The residue trace functionals are intimately related to
the spectral formulation of Wodzicki's noncommutative residue \cite{LocalInvarW, NCResidue1W}. In fact,
the formulation of the noncommutative residue as an integration over
the cosphere bundle of a manifold also is important for explicit computations
with noncommutative geometric spaces, see \cite{Residue4NCTF, ScalarCurv4NCTFK, NCResidueFW}
for a related treatment on noncommutative tori.

\looseness=1
The notion of a twisted spectral triple introduced by Connes and Moscovici~\cite{TrSpTypeIII} allows to incorporate a variety
of new examples, in particular type~III examples in the sense of the Murray--von~Neumann classif\/ication of operator algebras. They have shown that the Chern--Connes
character of a f\/initely summable twisted spectral triple is an ordinary cyclic cocycle
and enjoys an index pairing with $K$-theory. Also, they have constructed
a local Hochschild cocycle, which indicates that the ground is prepared for
extending the local index formula to the twisted case. This was carried out
in~\cite{TwsitedMos} for a particular class of twisted spectral triples; the analog
of Connes' character formula was investigated in~\cite{TwistedFatKha} for the
examples.
For treatments using twisted cyclic theory, in particular for relations
of the theory with Cuntz algebra~\cite{CuntzAlg} and quantum groups,
we refer to \cite{KMSCarNeshNes, TwistedCarPhiRen, CuntzAlgCarPhiRen},
see also \cite{ModularKaad, TwistedSU2KaaSen}. More recent works related
to the twisted version of spectral triples reveal their connections with the Bost--Connes
system, Riemann surfaces and graphs \cite{TwistedGreMarTeh}, and with the
standard model of particle physics \cite{TwistedDevMar}. Twisted spectral triples
associated with crossed product algebras are studied in \cite{TrSpTypeIII, TwistedIochum,TwsitedMos}, see also \cite{FatKhaTwistedSymbols} for an algebraic treatment.

Ergodic actions of compact groups on operator algebras are well-studied in the von Neumann setting (see, e.g., \cite{ErgodActIWassermann,ErgodActIIWassermann,ErgodClassSU2Wassermann}) and in a $C^*$-algebraic context. They were f\/irst introduced for $C^*$-algebras by E.~St\o{}rmer \cite{SpectraErgodTransfStormer} and this initial ef\/fort was expanded in various articles. Let us just mention two of them:
\begin{itemize}\itemsep=0pt
\item
In their article \cite{ErgodActAHK}, Albeverio and H\o{}egh-Krohn investigate in particular ergodic actions on commutative $C^*$-algebras $A = C(X)$ and prove that they correspond to continuous transitive actions on $X$.
\item
The article \cite{ErgodCpctGpHKLS} by H\o{}egh-Krohn, Landstad and St\o{}rmer proves that if $G$ acts ergodically on a unital $C^*$-algebra $A$, its unique $G$-invariant state is actually a \emph{trace}.
\end{itemize}
The article \cite{Rieffel98} was the f\/irst to suggest in 1998 that ergodic actions give rise to interesting spectral triples. This article proceeds with studying the metric induced on state spaces by ergodic actions. More recently, the article \cite{TrSpLieGpGG} produced a detailed construction of a so called \emph{Lie--Dirac operator} on a $C^*$-algebra $A$, based on an ergodic action of a compact Lie group~$G$ on~$A$. It also investigated the analytic properties of these Lie--Dirac operators, proving in particular that they are f\/initely summable spectral triples. In the present article, we elaborate on the techniques used in \cite{TrSpLieGpGG} in order to prove quite dif\/ferent results.

Indeed, in \cite{TrSpLieGpGG} the focus was on a Dirac operator for a ``noncommutative spin manifold'', whereas here
the emphasis is on a sort of Hodge--de~Rham operator associated with a conformally perturbed metric, construction of twisted spectral triples and the analog of the Chern--Gauss--Bonnet theorem. The Hodge--de~Rham operators constructed here are (in general) \emph{not} Lie--Dirac operators in the sense of \cite{TrSpLieGpGG}. In this previous article, the algebra structure of $A$ played only a~minor role in the analytical properties of the spectral triple. Here, the multiplication of $A$ has a central importance.

For a recent approach of Hodge theory using Hilbert modules, we refer to the recent ar\-tic\-le~\cite{EllipticComplexKrysl}.
See also~\cite{HodgeTheoryEllipticKrysl, HodgeTheoryKrysl}.

This article is organized as follows. In Section~\ref{Sec:Reminders},
we recall the necessary statements from representation theory and
operator theory, and the notion of ordinary and twisted spectral
triples along with their main properties that are used in our arguments
and concern our constructions. We associate a~complex of noncommutative dif\/ferential forms to a~$C^*$-dynamical system
$(A, G, \alpha)$ in Section~\ref{Sec:HodgeOperator}. In the ergodic case,
the analog of the Hodge--de~Rham operator is studied when the complex is equipped with
a Hermitian structure determined by a metric in the conformal class
of the standard metric associated with the unique $G$-invariant
trace on~$A$.

Inspired by a construction in~\cite{GaussBonnet-ConnesTretkoff},
we construct in Section~\ref{Sec:ConfTwistedTrSp} a spectral triple on~$A$
and a~twisted spectral triple
on the opposite algebra~$A^\text{op}$, which encode the geometric
information of the conformally perturbed metric. We study the
Dirac operator of the perturbed metric carefully and prove that
it is selfadjoint and enjoys having the same spectral dimension
as the non-perturbed case. It should be stressed that ergodicity
plays a crucial role for the latter to hold.

The existence of an analog of the Chern--Gauss--Bonnet theorem
is studied in Section \ref{Sec:CGBTheorem} by proving a
Hodge decomposition theorem for our complex and showing
that its Euler characteristic is independent of the conformal
factor. Combining this with the McKean--Singer index formula
and small time asymptotic expansions, which often exist for
noncommutative geometric spaces, we explain how the
analog of the Euler class or the Pfaf\/f\/ian of the curvature
form can be computed as local geometric invariants of
examples that f\/it into our setting. Indeed, such invariants
depend on the behavior at inf\/inity of the eigenvalues of the
involved Laplacians and the action of the algebra. Finally,
our main results and conclusions are summarized in
Section~\ref{Sec:Conclusions}.

\section{Preliminaries}
\label{Sec:Reminders}

We start by some reminders about results and notations from various anterior articles.
\begin{Definition}
Given a strongly continuous action $\alpha$ of a compact group $G$ on a unital $C^*$-algebra $A$, we say that it is \emph{ergodic} if the f\/ixed algebra of $G$-invariants elements is reduced to the scalars, i.e., if $\forall \, g \in G$, $\alpha_g(a) = a$, then $a \in \C 1_A$.
\end{Definition}

Among the important results obtained with this notion of ergodic action, let us quote the following \cite[Theorem~4.1, p.~82]{ErgodCpctGpHKLS}:
\begin{Theorem}
\label{Thm:ErgodAct}
Let $A$ be a unital $C^*$-algebra, $G$ a compact group and~$\alpha$ a strongly continuous representation of $G$ as an ergodic group of $*$-automorphisms of~$A$, then the unique $G$-invariant state~$\varphi_0$ on~$A$ is a~trace.
\end{Theorem}

Another result that will play an important role in our article is 
\cite[Proposition~2.1, p.~76]{ErgodCpctGpHKLS}, which we adapt slightly in the following:
\begin{Proposition}
\label{Prop:multiplicity}
Let $A$ be a unital $C^*$-algebra, $G$ a compact group and $\alpha$ a strongly continuous representation of~$G$ as an ergodic group of $*$-automorphisms of~$A$. Let~$V$ be an irreducible unitary representation of~$G$, $A(V)$ the spectral subspace of~$V$ in~$A$ and~$m(V)$ the multiplicity of~$V$ in~$A(V)$. Then we have
\begin{gather*}
m(V) \leqslant \dim V.
\end{gather*}
\end{Proposition}

Among our main results, we prove the f\/inite summability of certain spectral triples (ordinary and twisted), we therefore def\/ine those terms:
\begin{Definition}
\label{Def:TrSp}
Let $A$ be a unital $C^*$-algebra. An odd (ordinary) \emph{spectral triple}, also called an odd \emph{unbounded Fredholm module}, is a triple $(\Aa, \Hh, D)$ where
\begin{itemize}\itemsep=0pt
\item
$\Hh$ is a Hilbert space and $\pi \colon A\to B(\Hh)$ a $*$-representation of $A$ as bounded operators on~$\Hh$,
\item
$D$ is a selfadjoint unbounded operator -- which we will call the \emph{Dirac operator}~-- with domain~$\dom(D)$,
\end{itemize}
such that
\begin{enumerate}\itemsep=0pt
\item[(i)]
$(1+D^2)^{-1}$ is a compact operator,
\item[(ii)]
the subalgebra $\Aa$ of all $a\in A$ such that
\begin{gather*}
\pi(a) (\dom(D)) \ \subseteq \dom(D)
\qquad
 \text{and}
\qquad
[D,\pi(a)] \ \text{ extends to a bounded map on }\Hh
\end{gather*}
is dense in $A$.
\end{enumerate}
An \emph{even spectral triple} is given by the same data, but we further require that a grading $\gamma$ be given on $\Hh$ such that (i)~$A$ acts by even operators, (ii)~$D$ is odd.
\end{Definition}

\begin{Remark}
\label{Rk:SymmIdeals}
For a selfadjoint operator $D$, condition~(i) 
of the def\/inition above is actually equivalent to $\exists\, \lambda \in \R {\setminus} \{ 0 \}$ s.t.\ $(D+i \lambda)^{-1}$ is a compact operator.
\end{Remark}

To def\/ine \emph{finitely summable} spectral triples, we now need a brief reminder regarding \emph{trace ideals} (also known as \emph{symmetric ideals}), for which we follow Chapter~IV of~\cite{NCG}.
For more details concerning symmetrically normed operator ideals and singular traces we refer the reader to~\cite{TraceIdealsSimon} and~\cite{SingularTracesLSZ}.

\begin{Definition}
\label{Def:DixmierSum}
For $p>1$, the ideal $\SymmIdeal^{p^+}$ (also denoted $\SymmIdeal^{(p,\infty )}$ in \cite{NCG} and $\Jj_{p, \omega}$ in \cite[p.~21]{TraceIdealsSimon}) consists of all compact operators~$T$ on~$\Hh$ such that
\begin{gather*}
\| T\|_{p^+} := \sup_k \frac{ \sigma_k(T)}{k^{(p-1)/p}} < \infty,
\end{gather*}
where $\sigma_k$ is def\/ined as the supremum of the trace norms of~$T E$, when $E$ is an orthonormal projection of dimension~$k$, i.e.,
\begin{gather*}
 \sigma_k(T) := \sup \{ \| T E \|_1 , \dim E = k \}.
\end{gather*}
Equivalently, $\sigma_k(T)$ is the sum of the $k$ largest eigenvalues (counted with their multiplicities) of the positive compact operator $|T| := (T^* T)^{1/2}$. The def\/inition extends to the case of $p = 1$: $\SymmIdeal^{1^+}$~is the ideal of compact operators~$T$ s.t.\
\begin{gather*}
\| T \|_{1^+} := \sup_k \frac{ \sigma_k(T)}{\log k} < \infty .
\end{gather*}
The elements of $\SymmIdeal^{p^+}$ are called \emph{$p^+$-summable} (or $(p, \infty )$-summable -- see \cite[Section~IV.2$\alpha$, p.~299 and following]{NCG}).
\end{Definition}

A \emph{spectral dimension} for spectral triples is def\/ined as follows
\begin{Definition}
A spectral triple is \emph{$p^+$-summable} if $(1 + D^2)^{-1/2} \in \SymmIdeal^{p^+}$.
\end{Definition}

Finally, we will consider \emph{twisted spectral triples} (also called \emph{$\sigma$-spectral triples}) as introduced in \cite[Def\/inition~3.1]{TrSpTypeIII}. This is a spectral triple just like in Def\/inition~\ref{Def:TrSp}, but for a f\/ixed automorphism~$\sigma$ of~$\Aa$, the bounded commutators condition (denoted (ii) 
above) is replaced by
\begin{enumerate}\itemsep=0pt
\item[(ii)]
the subalgebra $\Aa$ of all $a\in A$ such that
\begin{itemize}\itemsep=0pt
\item
$\pi(a) (\dom(D)) \subseteq \dom(D)$,
\item
$D\pi(a) - \pi(\sigma(a)) D$ extends to a bounded map on $\Hh$
\end{itemize}
is dense in $A$.
\end{enumerate}

In this paper, we will need the subalgebras $\Aa^k$ for $k \geqslant 0$, corresponding to the $C^k$-dif\/fe\-ren\-tiab\-le class. Following \cite[Section~2.2]{DerivBratteli}, we introduce the space
\begin{gather*}
\Aa^m := \big\{ a \in A\colon g \mapsto \alpha_g(a) \text{ is in }C^m(G, A) \big\}.
\end{gather*}
Let us f\/ix a basis $(\partial _i)$ of the Lie algebra $\LieG$. For such a choice of basis, the inf\/initesimal genera\-tors~$\partial _i$ act as derivations $\Aa^m \to \Aa^{m-1}$. According to~\cite[Example~2.2.4, p.~41]{DerivBratteli}, $\Aa^m$ equipped with the norm
\begin{gather*}
\| a \|_m := \|a \| + \sum_{k=1}^m \sum_{i_1=1}^n \cdots \sum_{i_k=1}^n \frac{\| \partial_{i_1} \cdots \partial_{i_k}(a) \|}{k !},
\end{gather*}
is a Banach algebra with $\|a b \|_m \leqslant \| a \|_m \| b\|_m$. In particular, if $h \in \Aa^1$, then $e^{\lambda h} \in \Aa^1$, for all complex number~$\lambda$~-- see also Lemma~\ref{Lem:CVExp} below for a more precise estimate.

Following the density properties established in \cite{DerivBratteli} (see, e.g., Def\/inition~2.2.15, p.~47), the intersection $\Aa^\infty = \bigcap_{j=0}^\infty \Aa^j$ is a dense $*$-subalgebra of the $C^*$-algebra $A$, which is stable under the derivations $\partial _i$.

\section[Hodge--de~Rham Dirac operator and $C^*$-dynamical systems]{Hodge--de~Rham Dirac operator and $\boldsymbol{C^*}$-dynamical systems}
\label{Sec:HodgeOperator}

In this article, we consider a f\/ixed~$A$, a $C^*$-algebra with an ergodic action $\alpha$ of a compact Lie group~$G$ of dimension~$n$. We write~$\Aa$ for $\Aa^\infty$, the ``smooth subalgebra'' of~$A$, which can alternatively be def\/ined as
\begin{gather*}
\Aa := \{ a \in A \colon g \mapsto \alpha_g(a) \text{ is in }C^\infty (G, A) \}.
\end{gather*}
The Chevalley--Eilenberg cochain complex with coef\/f\/icients in $\Aa$ provides a complex that we interpret as ``dif\/ferential forms'' on $\Aa$. For the reader's convenience and to f\/ix notations, we provide a reminder of this construction. For all $k \in \N$,
\begin{gather*}
\Omega^k := \Aa \otimes \bigwedge^k \LieG^*,
\end{gather*}
where $\LieG^*$ denotes the linear forms on~$\LieG$, the Lie algebra of the Lie group~$G$. Given a scalar product on~$\LieG^*$ (e.g., obtained from the Killing form), we can extend it to a scalar product on~$\bigwedge^k \LieG^*$ by setting
\begin{gather*}
\langle v_1 \wedge \cdots \wedge v_k, w_1 \wedge \cdots \wedge w_k \rangle := \det( \langle v_i, w_j \rangle ),
\end{gather*}
i.e., the determinant of the matrix of scalar products. We f\/ix an orthonormal basis $(\omega_j)_{j = 1, \ldots , n}$ of $\LieG^*$ for this scalar product and consider its dual basis $(\partial _j)_{j=1, \ldots , n}$ in~$\LieG$.

Following \cite[the model of~(4.6), p.~157]{KnappLieGp}, we write the exterior derivative of the complex
\begin{gather}
d( a \otimes \omega_{i_1} \wedge \omega_{i_2} \wedge \cdots \wedge \omega_{i_K} ) =\sum_{j = 1}^n \partial _j(a) \otimes \omega_j \wedge \omega_{i_1} \wedge \omega_{i_2} \wedge \cdots \wedge \omega_{i_K} \nonumber\\
\qquad{}
 - \frac{1}{2} \sum_{k=1}^K \sum_{\alpha, \beta} (-1)^{k+1} c^{i_k}_{\alpha \beta} a \otimes \omega_\alpha \wedge \omega_\beta \wedge \omega_{i_1} \wedge \cdots \wedge \omega_{i_{k-1}} \wedge \omega_{i_{k+1}} \wedge \cdots \wedge \omega_{i_K},\label{Eqn:Defd}
\end{gather}
where $[\partial _i, \partial _j] = \sum\limits_{k=1}^n c^k_{i j} \partial _k$~-- the $c^{k}_{i j}$ are called the \emph{structure constants} of the Lie algebra~$\LieG$. A~lengthy but straightforward computation proves that this exterior derivative satisf\/ies $d^2 = 0$ on~$\Omega^\bullet$, therefore $(\Omega^\bullet , d)$ is a~complex.

\begin{Remark}
The Chevalley--Eilenberg complex is available even for noncompact groups $G$ and nonergodic actions. In other words, the square $d^2$ actually vanishes even when~$G$ is \emph{not} a~\emph{compact} Lie group and when the action of~$G$ on~$A$ is \emph{not} ergodic.
\end{Remark}

The natural product on $\Omega^\bullet$ is
\begin{gather}
\label{Eqn:ProdOmega}
(a \otimes v_1 \wedge \cdots \wedge v_k) \cdot (a' \otimes w_{1} \wedge \cdots \wedge w_{k'}) := a a' \otimes v_1 \wedge \cdots \wedge v_k \wedge w_1 \wedge \cdots \wedge w_{k'},
\end{gather}
i.e., the product of a $k$-form with a $k'$-form is a $k+k'$-form. In particular, for all~$k$, $\Omega^k$ is an $\Aa$-bimodule. The exterior derivative~$d$ is compatible with the right module structure in the following sense
\begin{gather*}
d( a a'\otimes \omega_{i_1} \wedge \omega_{i_2} \wedge \cdots \wedge \omega_{i_K} ) =\sum_{j = 1}^n \partial _j(a a') \otimes \omega_j \wedge \omega_{i_1} \wedge \omega_{i_2} \wedge \cdots \wedge \omega_{i_K} \\
\qquad\quad{}
 - \frac{1}{2} \sum_{k} \sum_{\alpha, \beta} c^{i_k}_{\alpha \beta} a a'\otimes \omega_\alpha \wedge \omega_\beta \wedge \omega_{i_1} \wedge \cdots \wedge \omega_{i_{k-1}} \wedge \omega_{i_{k+1}} \wedge \cdots \wedge \omega_{i_K}
\\
\qquad{}= \left( \sum_{j = 1}^n \partial _j(a) \otimes \omega_j \wedge \omega_{i_1} \wedge \omega_{i_2} \wedge \cdots \wedge \omega_{i_K} \right) a' \\
\qquad\quad{}
+\sum_{j = 1}^n a\partial _j(a') \otimes \omega_j \wedge \omega_{i_1} \wedge \omega_{i_2} \wedge \cdots \wedge \omega_{i_K} \\
\qquad\quad{}
 - \frac{1}{2} \left( \sum_{k} \sum_{\alpha, \beta} c^{i_k}_{\alpha \beta} a \otimes \omega_\alpha \wedge \omega_\beta \wedge \omega_{i_1} \wedge \cdots \wedge \omega_{i_{k-1}} \wedge \omega_{i_{k+1}} \wedge \cdots \wedge \omega_{i_K} \right) a'
\\
\qquad{}
= d( a \otimes \omega_{i_1} \wedge \omega_{i_2} \wedge \cdots \wedge \omega_{i_K} ) a'
+(-1)^K (a \otimes \omega_j \wedge \omega_{i_1} \wedge \omega_{i_2} \wedge \cdots \wedge \omega_{i_K}) d(a').
\end{gather*}
Since we want to treat conformal deformations of the original structure, we follow~\cite{GaussBonnet-ConnesTretkoff} and f\/ix a positive invertible element $e^{h} \in \Aa^1$, where~$h$ is a smooth selfadjoint element in~$\Aa^1$.
Then we def\/ine a scalar product on $\Omega^k$ by the formula
\begin{gather}
\label{Eqn:ScalProd}
(a \otimes v_1 \wedge \cdots \wedge v_k, a' \otimes w_1 \wedge \cdots \wedge w_k )_\varphi := \varphi_0\big( a^* a' e^{(n/2-k)h}\big) \det( \langle v_i, w_j \rangle ),
\end{gather}
where $\varphi_0$ is the unique $G$-invariant state on~$A$, which is actually a trace according to \cite[Theorem~3.1, p.~8]{ErgodActAHK}. We set the scalar product of two forms of dif\/ferent degrees to vanish. The scalar product obtained for $h = 0$ is the one we call the \emph{natural scalar product} on forms. We def\/ine the Hilbert space $\Hh_\varphi$ as the completion of $\Omega^\bullet$ for the scalar product~\eqref{Eqn:ScalProd}. In the particular case of $h = 0$, we obtain our reference Hilbert space $\Hh$. We will also need the Hilbert spaces $\Hh_{0, \varphi} := \GNS(\Aa, \varphi)$ and $\Hh_{0} := \GNS(\Aa, \varphi_0)$ as well as the Hilbert spaces $\Hh_k := \Hh_0 \otimes \bigwedge^k \LieG^*$~-- i.e., the completion of $k$-forms~-- and $\Hh_{k, \varphi}$.

To understand why we choose the form \eqref{Eqn:ScalProd} for the conformal deformation, we compare with the commutative case of a $n$-dimensional compact manifold $M$, where we have the following property: if the Riemannian metric is transformed by $g \rightsquigarrow \lambda g$ (for $\lambda > 0$), then the (pointwise) norm of all vectors is multiplied by~$\lambda^{1/2}$ and thus the pointwise norm of $1$-forms is multiplied by~$\lambda^{-1/2}$. This in turn implies that the pointwise norm of $k$-form is multiplied by~$\lambda^{-k/2}$. Finally, the (global) scalar product of $k$-forms is the integral of the pointwise scalar products. Since under the conformal deformation, the total volume of the manifold~$M$ is multiplied by~$\lambda^{n/2}$, the (global) scalar products of $k$-forms are multiplied by~$\lambda^{n/2-k}$. In particular, if~$n$ is even and $k = n/2$, then the scalar product on $n/2$-forms is left invariant under the conformal deformation.

In order to study $d$ and its adjoint, we introduce the degree $1$ maps $T_j \colon \bigwedge^\bullet \LieG^* \to \bigwedge^\bullet \LieG^*$ def\/ined for all $j \in \{1, \ldots, n\}$ by
\begin{gather*}
T_j( v_1 \wedge \cdots \wedge v_k) = \omega_j \wedge v_1 \wedge \cdots \wedge v_k.
\end{gather*}
Let $\RightMult_{x}$ denote the right multiplication operator for any $x \in A$: $\RightMult_{x}(a) = a x$, and let $B^{i_k}_{\alpha \beta}$ be the bounded operator on $\bigwedge^\bullet \LieG^*$ def\/ined using the basis $(\omega_j)$ by
\begin{gather*}
B^{i_k}_{\alpha \beta}( \omega_{i_1} \wedge \omega_{i_2} \wedge \cdots \wedge \omega_{i_K}) = \omega_\alpha \wedge \omega_\beta \wedge \omega_{i_1} \wedge \cdots \wedge \omega_{i_{k-1}} \wedge \omega_{i_{k+1}} \wedge \cdots \wedge \omega_{i_K}.
\end{gather*}
We can now give an explicit form to the operator $d$ and its (formal) adjoint for the unperturbed metric.
\begin{Lemma}
\label{Lem:Formed}
With the previous notations, when $h = 0$, the operator $d$ can be written
\begin{gather*}
d = \sum_j \partial_j \otimes T_j - \frac{1}{2} \sum_{k,\alpha,\beta} c^{i_k}_{\alpha \beta} \otimes B^{i_k}_{\alpha \beta}
\end{gather*}
and its adjoint $d^*$ is
\begin{gather*}
d^* = \sum_j \partial_j \otimes T_j^* - \frac{1}{2} \sum_{k,\alpha, \beta} \overline{c^{i_k}_{\alpha \beta}} \otimes \big(B^{i_k}_{\alpha \beta}\big)^*.
\end{gather*}
\end{Lemma}

\begin{proof}
The only point that is not self-explanatory is the behavior of $\partial_j$ with respect to the trace $\varphi_0$
\begin{gather*}
\langle [\partial_j(a)], [a'] \rangle = \varphi_0( \partial_j(a)^* a') = \varphi_0( -\partial_j(a^*) a') = \varphi_0(a^* \partial_j(a')),
\end{gather*}
where we used the relations $\partial_j(a)^* = - \partial_j(a^*)$ and $\varphi_0( \partial_j(a)) = 0$.
\end{proof}

\begin{Lemma}
\label{Lem:CVExp}
Let $h$ be an element of $\Aa^1$ and $\partial $ be an infinitesimal generator of $G$, acting as a~derivation on $\Aa^1$, then~$\partial ( e^{h} )$ is in the $C^*$-algebra $A$ and satisfies
\begin{gather*}
\big\| \partial \big( e^{h} \big) \big\| \leqslant \| \partial (h) \| e^{\| h\| }.
\end{gather*}
In particular, for a scalar parameter $v \to 0$, $\partial (e^{v h}) \to 0$.
\end{Lemma}

\begin{proof}
As an operator from $\Aa^1$ to $A$, the derivation $\partial $ is continuous, therefore we can estimate~$\partial (e^h)$ by using an $\Aa^1$-converging sequence, like the partial sums of~$e^{h}$.

For this sequence, using the derivation property, we get
\begin{gather*}
\left\| \partial \left( \sum_{k=0}^N \frac{h^k}{k !} \right) \right\| \leqslant \sum_{k=1}^N k \frac{\| \partial h \| }{k} \frac{\| h \|^{k-1}}{(k-1)!} \leqslant \| \partial h \| e^{\| h \|}.
\end{gather*}
The property $\partial (e^{v h}) \to 0$ as $v \to 0$ follows immediately.
\end{proof}

We call $d_\varphi$ the operator def\/ined on $\Hh_\varphi$ by the formula~\eqref{Eqn:Defd}. Once the scalar product~\eqref{Eqn:ScalProd} is def\/ined, we want to def\/ine an adjoint $d^*_\varphi$ to $d_\varphi$ for this scalar product. The Hodge--de~Rham operator that we would like to study \textit{in fine} is $d_\varphi + d_\varphi^*$. However, the unbounded operator $d_\varphi$ is \textit{a~priori} arbitrary, so it is not clear that it admits a densely def\/ined adjoint. To clarify the relations between~$\Hh$ and~$\Hh_\varphi$, we introduce the following lemma:
\begin{Lemma}
\label{Lem:EquivRep}
For any selfadjoint $h \in \Aa^1$,
\begin{itemize}\itemsep=0pt
\item
the Hilbert space $\Hh_{\varphi}$ is equipped with a $G$-representation defined on degree~$k$ forms by $\RepV_g( [a \otimes v_1 \wedge \cdots \wedge v_k]_\varphi) = [\alpha_g(a) \otimes v_1 \wedge \cdots \wedge v_k]_\varphi$, which leads to a $G$-equivariant left $\Aa$-module structures on $\Hh_{\varphi}$;
\item
the map $L \colon \Hh_0 \to \Hh_{0,\varphi}$ defined by $L([a]):= [a]_\varphi$ is invertible and intertwines the $G$-equivariant left $\Aa$-module structures on $\Hh_0$ and $\Hh_{0,\varphi}$.
\end{itemize}
Of course, $L$ extends to $L \otimes \id \colon \Hh \to \Hh_\varphi$ which is still a continuous and invertible map. We denote its adjoint by $H \colon \Hh_\varphi \to \Hh$, whose explicit form is
\begin{gather*}
H( [a \otimes v_1 \wedge \cdots \wedge v_k]_\varphi ) =\big [a e^{(n/2-k)h} \otimes v_1 \wedge \cdots \wedge v_k\big].
\end{gather*}
Finally, $\Hh$ and $\Hh_\varphi$ are related by the unitary map $U \colon \Hh \to \Hh_\varphi$ given on degree $k$ forms by
\begin{gather*}
U( [a \otimes v_1 \wedge \cdots \wedge v_k] ) = \big[a e^{-(n/2-k)h/2} \otimes v_1 \wedge \cdots \wedge v_k\big]_\varphi.
\end{gather*}
\end{Lemma}

\begin{Remark}
The map $U$ def\/ined above is unitary, but it does not intertwine the $G$-structures on $\Hh$ and $\Hh_\varphi$.
\end{Remark}

\begin{proof}
Since the sum $\Hh_{\varphi} = \bigoplus_k \Hh_{k, \varphi}$ is f\/inite, it suf\/f\/ices to check that $\RepV_g$ is continuous on each~$\Hh_{k, \varphi}$ separately. Since $\RepV_g$ does not act on $\bigwedge^\bullet \LieG^*$, it is enough to prove continuity on $\GNS(A, \tilde{\varphi})$ where $\tilde{\varphi}(a) = \varphi_0(a e^{-h_k})$ and $h_k = -(n/2 -k)h \in \Aa^1$ (corresponding to forms of degree~$k$). We get
\begin{gather*}
\| \alpha_g(a) \|_{\tilde{\varphi}}^2 = \varphi_0\big( \alpha_g(a)^* \alpha_g(a) e^{-h_k} \big) = \varphi_0\big( a^* a \alpha_{g^{-1}}\big(e^{-h_k}\big) \big) \\
\hphantom{\| \alpha_g(a) \|_{\tilde{\varphi}}^2}{}
= \varphi_0\big( a e^{-h_k/2} e^{h_k/2}\alpha_{g^{-1}}\big(e^{-h_k}\big)e^{h_k/2} e^{-h_k/2} a^*\big) \leqslant K \varphi_0\big( a e^{-h_k}a^*\big) = K \| a \|_{\tilde{\varphi}}^2,
\end{gather*}
for a constant $K = \| e^{h_k/2}\alpha_{g^{-1}}(e^{-h_k})e^{h_k/2} \|$. The above (scalar) inequality follows from the inequality of operators
\begin{gather*}
a e^{-h_k/2} e^{h_k/2}\alpha_{g^{-1}}\big(e^{-h_k}\big)e^{h_k/2} e^{-h_k/2} a^* \leqslant a e^{-h_k/2} K e^{-h_k/2} a^* = K a e^{-h_k} a^*,
\end{gather*}
which is valid since $e^{h_k/2}\alpha_{g^{-1}}(e^{-h_k})e^{h_k/2}$ is a positive operator. It then suf\/f\/ices to apply the positive functional~$\varphi_0$. Once we know that the map~$\RepV_g$ is def\/ined on the full Hilbert space, proving that it is compatible with the left $\Aa$-module structure is a formality.

The process is similar for $L$: it is clear from the def\/inition that if the map $L$ exists, then it intertwines the $G$-equivariant left $\Aa$-module structures on~$\Hh$ and~$\Hh_\varphi$. It remains to prove that~$L$ is well-def\/ined and invertible.

We f\/irst evaluate
\begin{gather*}
\| L( a) \|_{\tilde{\varphi}}^2 = \varphi_0\big( a^* a e^{-h_k}\big) = \varphi_0\big( a e^{-h_k} a^*\big) \leqslant \big\| e^{-h_k} \big\| \varphi_0(a a^*) = \big\| e^{-h_k} \big\| \, \| a \|_{\varphi_0}^2,
\end{gather*}
by the same argument as above.

To prove that $L$ is invertible, consider the norm of its inverse
\begin{gather*}
\| a \|_{\varphi_0}^2 = \varphi_0\big( a e^{-h_k/2} e^{h_k} e^{-h_k/2} a^*\big) \leqslant \big\| e^{h_k} \big\| \varphi_0\big( a e^{-h_k} a^*\big) = \big\|e^{h_k} \big\| \, \| a \|_{\tilde{\varphi}}^2.
\end{gather*}
The evaluation of the adjoint $H$ of~$L$ and of the unitary map $U \colon \Hh \to \Hh_\varphi$ is an easy exercise.
\end{proof}

With the previous notations, we see that $d_\varphi = (L \otimes \id) d (L \otimes \id)^{-1}$ and thus (at least formally) $d_\varphi^* = H^{-1} d^* H$. However, in order to facilitate the comparison between $d_\varphi + d^*_\varphi$ acting on $\Hh_\varphi$ and $d + d^*$ acting on~$\Hh$, we ``push'' $d_\varphi + d_\varphi^*$ to $\Hh$ using the unitary $U$. This leads us to the operators $D_u$ studied in the Proposition~\ref{Prop:Du} below. But f\/irst, for $h = h^*$, we need to introduce the operators $K_{u} \colon \Hh \to \Hh$, where $u \in \R$, def\/ined by
\begin{gather}
\label{Eqn:DefKu}
K_{u}( [a \otimes v_1 \wedge \cdots \wedge v_k ]) = \big[a e^{(n/2 - k)h u} \otimes v_1 \wedge \cdots \wedge v_k\big].
\end{gather}
This is a one-parameter group of invertible selfadjoint operators. Moreover, following our remark on $\Aa^1$ at the end of Section~\ref{Sec:Reminders}, for all $u \in \R$, the operators $K_u$ preserve the space $\Aa^1 \otimes \bigwedge^\bullet \LieG^*$.

In the proof below, we consider the orthogonal projections $\Pi_k \colon \Hh \to \Hh_k$ onto the completion of the $k$-forms, for all $k \in \{0, \ldots, n \}$.
\begin{Proposition}
\label{Prop:Du}
For all $u \in [0,1]$, we consider two unbounded operators defined on the dense domain $\Core = \Aa^1 \otimes \bigwedge^\bullet \LieG^* \subseteq \Hh$,
\begin{gather*}
d_u := K_{u} d K_{-u},
\qquad
d_u^* := K_{-u} d^* K_{u}.
\end{gather*}
We have:
\begin{enumerate}\itemsep=0pt
\item[$(1)$] 
if $E_{+,u}$ and $E_{-, u}$ are, respectively, the closures in $\Hh$ of the images of the operators $d_u$ and $d_u^*$, then $E_{+,u}$ and $E_{-, u}$ are orthogonal in $\Hh$; we denote by $\Pi_{+,u}$ and $\Pi_{-,u}$, respectively, the orthogonal projections on these spaces;
\item[$(2)$] 
the operator
\begin{gather*}
D_{u} = K_{u} d K_{-u} + K_{-u} d^* K_{u},
\end{gather*}
is essentially selfadjoint on a common core domain $\Core = \Aa^1 \otimes \bigwedge^\bullet \LieG^*$.
\end{enumerate}
The family of operators $D_u$ satisfies the estimate
\begin{gather}
\label{Eqn:EstimateDu}
\| (D_{u+v} - D_u)(\omegat) \| \leqslant o_v(1) \| D_u \omegat \| + o_v(1) \| \omegat \|,
\end{gather}
where $\omegat$ is any vector in the common selfadjointness domain and following Landau's notations, $o_v(1)$ stands for functions of $v$ which tend to $0$ for $v \to 0$. We can choose these two functions independently of the parameter $u \in [0,1]$.
\end{Proposition}

\begin{proof}
Regarding point~(1), 
we start by proving the property for $u = 0$, i.e., for the untwisted case. There, following Lemma~\ref{Lem:EquivRep} the trace $\varphi_0$ is $G$-invariant and therefore the action $\RepV_g$ of~$G$ on~$\Hh_0$ is unitary. Consequently, the $G$-representation can be decomposed into a direct sum of f\/inite-dimensional $G$-representations. Let us denote by~$V$ one of these f\/inite-dimensional spaces.

It is clear from the def\/inition \eqref{Eqn:Defd} that both $V \otimes \bigwedge^k \LieG^* \subseteq \Hh$ and its orthogonal are stable under the action of $d$. Thus the restriction of~$d$ to the f\/inite-dimensional space $V \otimes \bigwedge^k \LieG^*$ is bounded and admits an adjoint $d^*$ whose form is given by Lemma~\ref{Lem:Formed}. Varying the space $V$, we see that $d^*$ is def\/ined on $\Dd$, the algebraic direct sum of $V \otimes \bigwedge^k \LieG^*$, which is a dense subset of~$\Hh$. If we restrict to the case of forms $\omegat$, $\omegat'$ in the space $V \otimes \bigwedge^k \LieG^* \subseteq \Hh$, we have
\begin{gather*}
\langle d \omegat, d^* \omegat' \rangle = \langle d^2 \omegat, \omegat' \rangle = 0,
\end{gather*}
there are no considerations of domains for $d$ and $d^*$, since we consider f\/inite-dimensional spaces. The same argument applied to dif\/ferent f\/inite vector spaces $V$ proves that $E_{+,0}$ (the image of~$d$) and $E_{-,0}$ (the image of $d^*$) are orthogonal.

To treat point~(1) 
for a general $u \in [0,1]$, we note that $E_{+,u} = K_u E_{+,0}$: indeed, if $\xi = \lim d \omegat_n$, then $K_u \xi = \lim K_u d K_{-u} (K_u \omegat_n)$ and vice versa. Similarly, $E_{-,u} = K_{-u} E_{-,0}$. Given $K_u e_+ \in E_{+,u}$ and $K_{-u} e_- \in E_{-,u}$, we have
\begin{gather*}
\langle K_u e_+, K_{-u} e_- \rangle = \langle e_+, e_- \rangle = 0,
\end{gather*}
since $K_u$ is selfadjoint and $e_+ \in E_{+,0}$, $e_- \in E_{-,0}$ are orthogonal. This proves the requested orthogonality relation.

Regarding point~(2), 
let us start by giving a sketch of the proof: we f\/irst prove that $D_0 = D$ is essentially selfadjoint on the requested domain. Using the estimate~\eqref{Eqn:EstimateDu}, we then apply Kato--Rellich theorem to show that if $D_{u}$ is essentially selfadjoint for the domain $\Core$, then so is the operator $D_{u + v}$ for all $|v| \leqslant \varepsilon$, where~$\varepsilon$ is independent of the point $u \in [0,1]$ chosen. As a~consequence, all operators $D_u$ are essentially selfadjoint for the f\/ixed domain.

We f\/irst prove that $D_0$ is selfadjoint. This is done by using the Peter--Weyl decomposition of $\Hh_0$ for the unitary action $\RepV_g$ of $G$ on $\Hh_0$. As mentioned in point~(1), 
the restriction of $d$ to this f\/inite-dimensional space $V \otimes \bigwedge^k \LieG^* \subseteq \Hh$ is well-def\/ined, as is its adjoint $d^*$. Varying the space~$V$, we consider~$\Dd$, the direct sum of $V \otimes \bigwedge^k \LieG^*$, which is a dense subset of~$\Hh$.

In this situation, we can def\/ine $D = d + d^*$ on $\Dd$. If we restrict $D$ to a component $V \otimes \bigwedge^k \LieG^* \subseteq \Hh$, it is formally selfadjoint by def\/inition. It therefore admits an orthonormal basis of eigenvectors with real associated eigenvalues. It follows that $\Ran(D +i)$ and $\Ran(D -i)$ are dense in~$\Hh$, and this is enough to prove that $D$ is essentially selfadjoint on the domain $\Core$ (see \cite[Corollary, p.~257]{ReedSimon1}).

Let us now consider an arbitrary $u \in [0,1]$. We want to f\/ind $\varepsilon > 0$ uniform in $u$ and small enough so that for all $v$ with $|v| \leqslant \varepsilon$, the operator $D_{u+v}$ is essentially selfadjoint. By def\/inition,
\begin{gather*}
D_{u+v} = K_{u+v} d K_{-(u+v)} + K_{-(u+v)} d^* K_{u+v}.
\end{gather*}
If we introduce $R_v = K_v -1$, then we can write
\begin{gather*}
K_{u+v} = K_u (1 + R_v),
\qquad
K_{-(u+v)} = (1 + R_{-v}) K_{-u}.
\end{gather*}
It is clear that both $R_v$ and $R_{-v}$ are bounded with $\|R_v \| \to 0$, $\|R_{-v} \| \to 0$ for $v \to 0$ and by def\/inition, for all $v \in \R$, $R_v$ commutes with~$K_u$ for $u \in \R$.

We write
\begin{gather*}
D_{u+v} = K_{u} (1 + R_v) d (1+ R_{-v}) K_{-u} + K_{-u} (1+R_{-v}) d^* (1+R_v) K_{u} \\
\hphantom{D_{u+v}}{}
= K_u d K_{-u} + K_u d R_{-v} K_{-u} + K_u R_v d K_{-u} + K_u R_v d R_{-v} K_{-u} \\
\hphantom{D_{u+v}=}{}
+ K_{-u} d^* K_{u} + K_{-u} d^* R_v K_u + K_{-u} R_{-v} d^* K_u + K_{-u} R_{-v} d^* R_v K_u.
\end{gather*}
The sum of the terms $K_u d K_{-u}$ and $K_{-u} d^* K_{u}$ gives back $D_u$. Since $E_{+, u}$ and $E_{-,u}$ are orthogonal, we have
\begin{gather}
\label{Eqn:Pyth}
\| D_u \omegat \|^2 = \| \Pi_{+, u}( D_u \omegat ) \|^2 + \| \Pi_{-, u}( D_u \omegat ) \|^2 = \| d_{u} \omegat \|^2 + \| d_u^* \omegat \|^2.
\end{gather}
Both $D_{u+v}$ and $D_u$ are symmetric operators, so their dif\/ference (namely the sum~$\Sigma$ of the six remaining terms) is a symmetric operator. By hypothesis, $D_u$ is selfadjoint. According to Kato--Rellich theorem as stated in \cite[Theorem~X.12, p.~162]{ReedSimon2}, it therefore only remains to prove that~$\Core$ is also a domain for~$\Sigma$ and that for all $\omegat \in \Core$,
\begin{gather*}
\| \Sigma(\omegat) \| \leqslant a \| D_u \omegat \| + b \| \omegat \|,
\end{gather*}
where both real numbers $a$, $b$ are positive and $a < 1 $. It is clear from the def\/inition of~$K_{\pm u}$ and~$R_v$ that their actions preserve the core $\Core$ of $C^1$-functions on $G$ and thus $\Sigma(\omegat)$ has a well-def\/ined meaning for all $\omegat \in \Core$. We decompose $\omegat \in \Core$ into a~sum $\omegat = \sum_{k} \omegat_k$ of $C^1$-forms of degree $k$ and start by an estimate of the dif\/ferent terms $\| \Sigma \omegat_k \|$ for any f\/ixed~$k$.

Remember from Lemma \ref{Lem:Formed} that $d$ can be written
\begin{gather*}
d = \sum_j \partial_j \otimes T_j - \frac{1}{2} \sum_{k,\alpha,\beta} c^{i_k}_{\alpha \beta} \otimes B^{i_k}_{\alpha \beta},
\end{gather*}
where the dif\/ferent $B^{i_k}_{\alpha \beta}$ are bounded operators. We remark that the $B^{i_k}_{\alpha \beta}$ commute with right multiplications, like the one appearing in the def\/inition of $K_u$ acting on an element of given degree. For $\omegat = a \otimes \vt$ of degree $k$, we write
\begin{gather*}
K_u d R_{-v} K_{-u} a \otimes \vt = \sum \partial _j(a e^{-(n/2-k)h u}(e^{-(n/2-k) hv} -1) )e^{(n/2-(k+1)) h u} \otimes T_j( \vt ) \\
\hphantom{K_u d R_{-v} K_{-u} a \otimes \vt =}{}
- \frac{1}{2} \sum_{k} \sum_{\alpha, \beta} c^{i_k}_{\alpha \beta} a e^{-(n/2-k)h u}(e^{-(n/2-k) hv} -1) e^{(n/2-(k+1)) h u} \otimes B^{i_k}_{\alpha \beta} \vt \\
\hphantom{K_u d R_{-v} K_{-u} a \otimes \vt }{}
= \sum \partial _j(a e^{-(n/2-k)h u})e^{(n/2-(k+1)) h u} (e^{-(n/2-k) hv} -1) \otimes T_j( \vt ) \\
\hphantom{K_u d R_{-v} K_{-u} a \otimes \vt =}{}
- \frac{1}{2} \sum_{k} \sum_{\alpha, \beta} c^{i_k}_{\alpha \beta} a e^{-h u} (e^{-(n/2-k) hv} -1) \otimes B^{i_k}_{\alpha \beta} \vt \\
\hphantom{K_u d R_{-v} K_{-u} a \otimes \vt =}{}
+ \sum a e^{-(n/2-k)h u}\partial _j(e^{-(n/2-k) hv} -1) )e^{(n/2-(k+1)) h u} \otimes T_j( \vt ) \\
\hphantom{K_u d R_{-v} K_{-u} a \otimes \vt }{}
= R_{-v}(K_u d K_{-u})(a \otimes \omegat) \! + K_{u} \circ\! \sum_{j} (\RightMult_{\partial _j(e^{-(n/2-k) h v})}\! \otimes T_j) \circ K_{-u}(a \otimes \vt).
\end{gather*}
Taking a linear combination to treat the case of a sum $\omegat = \sum \omegat_k$, we get
\begin{gather*}
K_u d R_{-v} K_{-u} = R_{-v}(K_u d K_{-u}) + K_{u} \circ \bigg( \sum_{j,k} (\RightMult_{\partial _j(e^{-(n/2-k) h v})} \otimes T_j) \circ \Pi_k \bigg) \circ K_{-u}.
\end{gather*}
In this equality, $\sum_{j,k} (\RightMult_{\partial _j(e^{-(n/2-k) h v})} \otimes T_j) \circ \Pi_k$ is a f\/inite sum of bounded operators. As a~consequence of Lemma~\ref{Lem:CVExp}, the norm of these operators tend to $0$ for $ v \to 0$. We already know that $R_{-v}$ tends to $0$ in norm for $v \to 0$, we therefore get the estimate
\begin{gather}
\label{Eqn:EstimateKdRK}
\| K_u d R_{-v} K_{-u}( \omegat) \| \leqslant o_v(1) \| K_u d K_{-u}( \omegat) \| + o_v(1) \| \omegat \|.
\end{gather}
The two functions $o_v(1)$ can be taken uniform in $u \in [0,1]$, since $[0,1]$ is a compact.

The term $K_u R_v d K_{-u}$ is easily treated: $K_u R_v d K_{-u} = R_v K_u d K_{-u}$. The term $K_u R_v d R_{-v} K_{-u}$ is processed similarly: $ K_u R_v d R_{-v} K_{-u} = R_v K_u d R_{-v} K_{-u}$ and then the estimate \eqref{Eqn:EstimateKdRK} enables us to write
\begin{gather*}
\| K_u R_v d R_{-v} K_{-u}(\omegat) \| \leqslant o_v(1) \big\| \big(K_u d K_u^{-1}\big)(\omegat) \big\| + o_v(1) \| \omegat \|.
\end{gather*}
As a result, we get
\begin{gather}
\label{Eqn:Estd}
\| (d_{u+v} - d_u) \omegat \| \leqslant o_v(1) \| d_u \omegat \| + o_v(1) \| \omegat \|.
\end{gather}
Lemma \ref{Lem:Formed} af\/fords a similar treatment of the term $K_{-u} d^* R_v K_u $, just replacing $T_j$ by $T^*_j$, $B^{i_k}_{\alpha \beta}$ by $(B^{i_k}_{\alpha \beta})^*$ and $c^{i_k}_{\alpha \beta}$ by $\overline{c^{i_k}_{\alpha \beta}}$. We get an estimate
\begin{gather}
\label{Eqn:EstdAdj}
\| (d_{u+v}^* - d_u^*) \omegat \| \leqslant o_v(1) \| d_u^* \omegat \| + o_v(1) \| \omegat \|.
\end{gather}
Using the equation \eqref{Eqn:Pyth}, which ensures that $\| d_u \omegat \| \leqslant \| D_u \omegat \|$ and $\| d_u^* \omegat \| \leqslant \| D_u \omegat \|$, we can combine~\eqref{Eqn:Estd} and~\eqref{Eqn:EstdAdj} to show that the relation \eqref{Eqn:EstimateDu} is satisf\/ied.

We can therefore apply the Kato--Rellich theorem for all $u \in [0,1]$ and this proves that all~$D_u$ (including~$D_1$) have selfadjoint extensions with the same core $\Core = \Aa^1 \otimes \bigwedge^\bullet \LieG^*$.
\end{proof}

\begin{Remark}
\label{Rk:Domain}
It appears from the proof of point~(2) 
that we could also take $\Aa^\infty \otimes \bigwedge^\bullet \LieG^*$ as core for the operator $D_0$ (using the Peter--Weyl decomposition). If we further assume $h \in \Aa^\infty $, the rest of the proof applies \textit{verbatim} and shows that all $D_u$ have a common core, namely $\Aa^\infty \otimes \bigwedge^\bullet \LieG^*$.
\end{Remark}

\begin{Corollary}
\label{Cor:FiniteSumma}
For all selfadjoint elements $h \in \Aa^1$ and all parameters $u \in [0,1]$, the opera\-tors~$D_u$ are $n^+$-summable.
\end{Corollary}

\begin{proof}
In the untwisted case, i.e., for $D_0$, we can follow the argument of Theorem~5.5 of \cite{TrSpLieGpGG} to prove that $d +d^*$ is $n^+$-summable, where $n$ is the dimension of~$G$. Indeed, according to Proposition~\ref{Prop:multiplicity} from~\cite{ErgodCpctGpHKLS}, as $G$-vector spaces, we have $\Hh \inj \Hh_\text{ref}$ where $\Hh_\text{ref} := L^2(G) \otimes \bigwedge^k \LieG^*$. Moreover, the operator $d +d^* := D_\text{ref}$ on this space is just the Hodge--de~Rham operator on~$G$ and therefore it is $n^+$-summable. Since $D_\text{ref}$ also preserves the f\/inite-dimensional spaces $V \otimes \bigwedge^k \LieG^*$ obtained by Peter--Weyl decomposition, the eigenvalues of $|D|$ coincide with those of $|D_\text{ref}|$ except that they may have lower (and possibly zero) multiplicities. Consequently, the same computation as in~\cite{TrSpLieGpGG} proves that~$D$ is $n^+$-summable.

To extend this property to all $D_u$ for $u \in [0,1]$, we f\/irst note that to prove $D_u$ is $n^+$-summable, it suf\/f\/ices to show that the operator $(D_u + i)^{-1}$ is in the symmetric ideal $\SymmIdeal^{n^+}$ -- as mentioned in Remark \ref{Rk:SymmIdeals}. The existence of the operator $(D_u + i)^{-1}$ is a consequence of Proposition \ref{Prop:Du}. The discussion above proves that $(D_0 + i)^{-1}$ is in this ideal.

We then use \cite[Theorem~1.16, p.~196]{Kato} to prove that if $(D_u + i)^{-1} \in \SymmIdeal^{n^+}$ then for some $\varepsilon > 0$ small enough but independent of $u \in [0,1]$, and for any $v$ in $|v| \leqslant \varepsilon$, then $(D_{u+v} + i)^{-1} \in \SymmIdeal^{n^+}$. For all $u$, $v$,
\begin{gather*}
(D_{u+v} + i ) - (D_u + i) = D_{u+v} - D_u,
\end{gather*}
and to apply Kato's stability property, we need to give a relative bound on $D_{u+v} - D_u$, expressed in terms of $D_u +i$. We are going to obtain this using the relation~\eqref{Eqn:EstimateDu}. Indeed, since we know that $D$ is selfadjoint, $\langle D \xi, \xi \rangle = \langle \xi, D \xi \rangle $ and thus
\begin{gather*}
\| (D + i) \xi \|^2 = \| D \xi \|^2 + \| \xi \|^2.
\end{gather*}
which shows that $\| D \xi \| \leqslant \| (D + i) \xi \|$. From this fact and \eqref{Eqn:EstimateDu}, we deduce
\begin{gather*}
\| (D_{u+v} - D_u)(\omegat) \| \leqslant o_v(1) \| (D_u +i) \omegat \| + o_v(1) \| \omegat \|,
\end{gather*}
which let us apply \cite[Theorem~1.16, p.~196]{Kato} to $D_u +i$ and $D_{u+v} - D_u$, leading to the expression
\begin{gather*}
(D_{u+v} + i)^{-1} = (D_u + i)^{-1} \big(1 + (D_{u+v} - D_u)(D_u +i)^{-1}\big)^{-1},
\end{gather*}
where both $(D_{u+v} - D_u)(D_u +i)^{-1}$ and $(1 + (D_{u+v} - D_u)(D_u +i)^{-1})^{-1}$ are bounded operators. This expression shows that $(D_{u+v} + i)^{-1}$ is a product of $(D_u + i)^{-1}$ in the ideal $\SymmIdeal^{n^+}$ and a~bounded operator. It is therefore itself in the ideal $\SymmIdeal^{n^+}$ and this completes the proof.
\end{proof}

The operator $d_u$ of Proposition \ref{Prop:Du} induces a cochain complex:
\begin{Proposition}
The operator $d_u := K_u d K_{-u}$, defined from $d = \Pi_+ D$ on the domain of selfadjointness of $D_u$ is closable. Taking its closure, there is a cochain complex $\Complex[u]$
\begin{gather}
\label{Eqn:Complex}
0 \to \Hh_{0} \xrightarrow{d_{u,0}} \Hh_1 \to \cdots \to \Hh_{n-1} \xrightarrow{d_{u, n-1}} \Hh_n \to 0.
\end{gather}
\end{Proposition}

\begin{Remark}
In the complex \eqref{Eqn:Complex}, the map $d_{u,k} \colon \Hh_{k} \to \Hh_{k+1}$ is of course (the closure of) the restriction of~$d_u$ to $\Hh_k \cap \dom(D_u)$, where $\dom(D_u)$ is the domain of selfadjointness of~$D_u$.
\end{Remark}

\begin{proof}
We f\/irst treat the case of $d$ (for $h = 0$). In this case, if $x_n \to x$ and $y_n \to x$ while both~$d x_n$ and~$d y_n$ converge, we want to prove that $\lim d x_n = \lim d y_n$. Consider any $z \in \Hh$ which lives in a f\/inite-dimensional vector space $V \otimes \bigwedge^\bullet \LieG^*$ obtained from the Peter--Weyl decomposition. This ensures that $\Pi_+ z$ is in $V \otimes \bigwedge^\bullet \LieG^*$ and thus in the domain of $D$. We then have
\begin{gather*}
\langle z, \Pi_+ D x_n \rangle = \langle D \Pi_+ z, x_n \rangle \to \langle D \Pi_+ z, x \rangle \leftarrow \langle z, \Pi_+ D y_n \rangle.
\end{gather*}
Since we know that both $d x_n$ and $d y_n$ converge in $\Hh$ and that~$\Dd$, the algebraic direct sum of all $V \otimes \bigwedge^\bullet \LieG^*$ is dense, it is necessary that $\lim d x_n = \lim d y_n$ and this proves that~$d$ is closable.

It follows that the kernel $\ker( \overline{d})$ is closed. Since $\Aa^1 \otimes \bigwedge^\bullet \LieG^*$ is a core for $D$, any $x$ in the domain $\dom( \overline{d} )$ can be approximated by $x_n \in \Aa^1 \otimes \bigwedge^\bullet \LieG^*$ such that $x_n \to x$ and $d x_n \to d x$. The density of $\Aa^\infty $ inside $\Aa^1$ (as discussed at the end of Section~\ref{Sec:Reminders}) then provides an approximation of the original $x \in \dim( \overline{d} )$ by $y_n \in \Aa^\infty \otimes \bigwedge^\bullet \LieG^* = \Omega^\bullet$. For this sequence~$y_n$, we know from Section~\ref{Sec:HodgeOperator} that $d^2 y_n = 0$. By density, we obtain that~\eqref{Eqn:Complex} is a cochain complex.

Similarly, for $d_u = K_u d K_{-u}$ if $x_n \to x$ and $y_n \to x$ while both $d_u x_n$ and $d_u y_n$ converge, we have $K_{-u} x_n \to K_{-u} x \leftarrow K_{-u} y_n$ and $K_{-u} d_u x_n = d K_{-u} x_n$, $K_{-u} d_u y_n = d K_{-u} y_n$. Since $d$ is closable, we get $\lim K_{-u} d_u x_n = \lim K_{-u} d_u y_n$, which suf\/f\/ices to prove that $d_u$ is also closable. The cochain property then follows from $d_u^2 = K_u d^2 K_{-u} = 0$.
\end{proof}
In the rest of this section, we will be interested in the \emph{reduced cohomology} of the complex \eqref{Eqn:Complex}, namely the cohomology groups
\begin{gather}
\label{Eqn:RedCohom}
H^k\Complex[u] := \ker(d_{u,k})/ \overline{\Ran(d_{u,k-1})}.
\end{gather}

For any $u \in [0,1]$, let us write $E_{0,u}$ for the kernel of~$D_u$. We have the following Hodge decomposition theorem for the conformally perturbed metric:
\begin{Theorem}
\label{Thm:HodgeDecomp}
Let $G$ be a \emph{compact} Lie group of dimension $n$ acting \emph{ergodically} on a unital $C^*$-algebra $A$. With the notations introduced previously, for any parameter $u \in [0,1]$, there is a~decomposition of $\Hh$ into a~direct sum of orthogonal Hilbert spaces
\begin{gather*}
\Hh = E_{-,u} \oplus E_{0,u} \oplus E_{+,u}.
\end{gather*}
\end{Theorem}

\begin{proof}
The operator $D_u$ is selfadjoint with compact resolvent, as a consequence of Proposition~\ref{Prop:Du} and Corollary~\ref{Cor:FiniteSumma}. Thus, we have an orthogonal sum $\Hh = E_{0,u} \oplus \overline{\Ran(D_u)}$. Following Proposition~\ref{Prop:Du}, $\overline{\Ran(D_u)} = E_{-,u} \oplus E_{+,u}$ and the sum is orthogonal, which proves the result.
\end{proof}

We call the restriction of $D_u^2$ to $\Hh_k$ the \emph{Laplacian on $\Hh_k$} and denote it by $\Delta_k$, which is thus an unbounded operator on~$\Hh_k$, def\/ined on the domain $\Aa^\infty \otimes \bigwedge^k \LieG^*$. Note that $\Delta_k$ actually depends on our choice of conformal perturbation $h \in \Aa^1$.

\begin{Corollary}
\label{Cor:CohomGp}
Let $H^k\Complex[u]$ be the cohomology groups introduced in~\eqref{Eqn:RedCohom}, they identify naturally with the kernel of~$\Delta_k$, i.e.,
\begin{gather*}
\ker(\Delta_k) \simeq H^k\Complex[u].
\end{gather*}
\end{Corollary}

\begin{Remark}
\label{remarkfinitedimensionalityofcoh}
This Corollary implies in particular that these cohomology groups are f\/inite-dimensional, since $\ker(\Delta_k) = \ker(D_u)$ and $D_u$ has compact resolvent by Corollary~\ref{Cor:FiniteSumma}.
\end{Remark}

\begin{proof}
The cohomology group $H^k\Complex[u]$ is def\/ined as $\ker(d_{u, k})/\overline{\Ran(d_{u, k-1})}$. The Hodge decomposition Theorem \ref{Thm:HodgeDecomp} can be combined with the projections $\Pi_\pm$ and $\Pi_k$ on $\Hh_k$ to prove that $\Hh_k = \ker(\Delta_k) \oplus \overline{\Ran(d_{u,k-1})} \oplus \overline{\Ran(d_{u, k}^*)}$. We know that $\ker(d_{u, k}) = \Ran(d_{u,k}^*)^\perp$. Therefore $\ker(d_{u,k}) = \overline{\Ran(d_{u,k-1})} \oplus \ker(\Delta_k)$ from which it follows immediately that $H^k\Complex[u] = \ker(d_{u,k})/\overline{\Ran(d_{u,k-1})} \simeq \ker(\Delta_k)$.
\end{proof}

\begin{Proposition}
\label{Prop:Stability}
The cohomology groups $H^k\Complex[u]$ are abstractly isomorphic to the nonperturbed $(h = 0)$ cohomology groups $H^k\Complex$.
\end{Proposition}

\begin{proof}
It is easy to check that $\ker(d_u) = K_u \ker(d)$ and $E_{+, u} = K_u E_{+,0}$. Thus, as abstract vector space, $\ker(d_u)/E_{+,u} = K_u \ker(d_0)/ K_u E_{+,0}$ is f\/inite-dimensional, with the same dimension as $\ker(d_0)/E_{+,0}$.
\end{proof}

\begin{Remark}
The dimensions of $\ker(d_u)/E_{+,u}$ and $\ker(d_0)/E_{+,0}$ are the same, but there are not ``{concretely} isomorphic'' for the scalar product we consider. The \emph{concrete realisation} of $\ker(d_u)/E_{+, u}$ is $\{ \omegat \in \ker(d_u)\colon \forall\, \omegat' \in E_{+,u}, \langle \omegat, \omegat' \rangle = 0 \}$. However, $K_u$ does not preserve scalar products and therefore, $\ker(d_u)/E_{+,u}$ is not realised concretely by $K_u E_{0,0}$. In other words, $K_u E_{0,0}$ is \emph{not} the space of harmonic forms for $D_u$.
\end{Remark}

\section[Conformally twisted spectral triples for $C^*$-dynamical systems]{Conformally twisted spectral triples\\ for $\boldsymbol{C^*}$-dynamical systems}
\label{Sec:ConfTwistedTrSp}

In the following theorem, we use the selfadjoint operator $D_u$ to construct spectral triples for the natural actions of the algebra $A$ (with its left action on $\Hh$) and the algebra $A^\text{op}$ (acting on the right of $\Hh$).
\begin{Theorem}
\label{Thm:Main}
Let $G$ be a \emph{compact} Lie group of dimension $n$ acting ergodically on a~unital $C^*$-algebra $A$, then using the unique $G$-invariant trace $\varphi_0$ of Theorem~{\rm \ref{Thm:ErgodAct}}, we write $\Hh_0 := \GNS(A, \varphi_0)$.

For any fixed $h \in \Aa^1$ and any $u \in [0,1]$, the data $(A, \Hh_0 \otimes \bigwedge^\bullet \LieG^*, D_u)$ with grading $\gamma$ defines an even $n^+$-summable spectral triple, where
\begin{itemize}\itemsep=0pt
\item
the representation $\pi$ of $A$ on $\Hh = \Hh_0 \otimes \bigwedge^\bullet \LieG^*$ is given by restriction of the left multiplication~\eqref{Eqn:ProdOmega};
\item
the unbounded operator $D_u$ is the unique selfadjoint extension of
\begin{gather*}
D_{u} = K_{u} d K_{-u} + K_{-u} d^* K_{u},
\end{gather*}
defined on the core $\Core = \Aa^1 \otimes \bigwedge^\bullet \LieG^*$, the operator $K_u$ being defined by~\eqref{Eqn:DefKu};
\item
the grading operator $\gamma$ is defined on degree $k$ forms by
\begin{gather*}
\gamma( a \otimes v_1 \wedge \cdots \wedge v_k) = (-1)^k ( a \otimes v_1 \wedge \cdots \wedge v_k).
\end{gather*}
\end{itemize}
For any fixed $h \in \Aa^1$ and any $u \in [0,1]$, the data $(A^\text{op}, \Hh_0 \otimes \bigwedge^\bullet \LieG^*, D_u)$ with grading $\gamma$ defines an even $n^+$-summable \emph{twisted} spectral triple, with the automorphism $\beta$ on $A$ given by $\beta(a) = e^{h u} a e^{-h u}$ -- we use this $\beta$ to define an automorphism on $A^\text{op}$.
\end{Theorem}

\begin{Remark}
The morphism $\beta$ def\/ined above preserves the multiplication of $A^\text{op}$. It also satisf\/ies the relation \emph{unitarity condition} (see \cite[equation~(3.4)]{TrSpTypeIII}) that is $\beta\big( (a^\text{op})^* \big) = (\beta^{-1}(a^\text{op}))^*$.
\end{Remark}

\begin{proof}
It is clear from the def\/inition of $\pi$ that $A$ is represented on $\Hh$ by bounded operators. The existence and uniqueness of the selfadjoint extension of $D_u$ is proved in Proposition~\ref{Prop:Du}, while the compact resolvent and f\/inite summability properties are shown in Corollary~\ref{Cor:FiniteSumma}.

We now prove that the commutator of $D_u$ with $a \in \Aa^1$ is bounded. To this end, we use the notations of Lemma~\ref{Lem:Formed} to decompose the operator $D_u$. We call
\begin{itemize}\itemsep=0pt
\item
Part (0) is the ``bounded part'' of $D_u$, that is the terms
\begin{gather*}
 - K_u \bigg( \frac{1}{2} \sum_{k,\alpha,\beta} c^{i_k}_{\alpha \beta} \otimes B^{i_k}_{\alpha,\beta} \bigg) K_{-u}
\qquad
\text{and}
\qquad
- K_{-u} \bigg( \frac{1}{2} \sum_{k,\alpha,\beta} \overline{c^{i_k}_{\alpha, \beta}} \otimes (B^{i_k}_{\alpha, \beta})^* \bigg) K_{u}.
\end{gather*}
\item
Part (I) consists of the terms
\begin{gather*}
K_u \bigg( \sum_{j} \partial _j \otimes T_j \bigg) K_{-u}.
\end{gather*}
\item
Part (II) consists of the terms
\begin{gather*}
K_{-u} \bigg( \sum_{j} \partial _j \otimes T_j^* \bigg) K_{u}.
\end{gather*}
\end{itemize}
Part (0) commutes with the left multiplication by $a \in \Aa^1$, and thus it does not contribute to the commutator. We therefore only need to estimate Parts~(I) and~(II) of $D_u( a' \omegat)$ for $\omegat = a \otimes v_1 \wedge \cdots \wedge v_k$, that is
\begin{gather*}
\sum_j \partial_j\big( a' a e^{-(n/2 - k)h u} \big) e^{(n/2 - (k+1))h u} \otimes T_j( v_1 \wedge \cdots \wedge v_k) \\
\qquad\quad{}
+ \sum_j \partial_j\big( a' a e^{(n/2 - k)h u} \big) e^{-(n/2 - (k-1))h u} \otimes T_j^*( v_1 \wedge \cdots \wedge v_k)\\
\qquad{}
= \sum_j \big(\partial_j( a') a e^{-(n/2 - k)h u} + a' \partial_j\big(a e^{-(n/2 - k)h u}\big)\big) e^{(n/2 - (k+1))h u}
\otimes T_j( v_1 \wedge \cdots \wedge v_k) \\
\qquad\quad{}
+ \sum_j \big(\partial_j( a') a e^{(n/2 - k)h u} + a' \partial_j\big(a e^{(n/2 - k)h u}\big)\big) e^{-(n/2 - (k-1))h u} \otimes T_j^*( v_1 \wedge \cdots \wedge v_k).
\end{gather*}
It follows from these considerations that
\begin{gather*}
[D_u, a'] \omegat = \sum_j \partial_j(a') a e^{-h u} \otimes (T_j + T_j^*)( v_1 \wedge \cdots \wedge v_k),
\end{gather*}
which is clearly a bounded function of $\omegat$ for any $a' \in \Aa^1$. Moreover, such $a' \in \Aa^1$ sends the core~$\Core$ of our selfadjoint operator $D_u$ to itself and following \cite[Proposition~A.1, p.~293]{TrSpPiCr-Paterson}, this suf\/f\/ices to ensure that $a' \in \Aa^1$ sends the domain of~$D_u$ to itself. The algebra $\Aa$ of Def\/inition~\ref{Def:TrSp} thus contains $\Aa^1$ and is dense in the $C^*$-algebra $A$. This completes the proof that $(A, \Hh, D_u)$ is a~$n^+$-summable spectral triple.

\looseness=-1
It remains to study its parity: it is clear from the def\/inition that $\gamma$ sends the core $\Core$ to itself and thus it leaves the full domain of the selfadjoint operator $D_u$ stable. Clearly, $\gamma$ distinguishes only between $\Hh_\text{even} := A \otimes \bigwedge^\text{even} \LieG^*$ and $\Hh_\text{odd} := A \otimes \bigwedge^\text{odd} \LieG^*$ and $\pi(a)$ leaves both spaces invariant, while $D_u$ is an odd operator. This proves that $(A, \Hh, D_u)$ with $\gamma$ is an even spectral triple.

The parity paragraph above applies \textit{verbatim} to the spectral triple constructed from the right action of $A^\text{op}$. The summability property is also conserved. It remains to investigate the bounded twisted commutators. Notice f\/irst that if $a', h \in \Aa^1$ then both right multiplications by~$a'$ and by~$\beta(a')$ leave the core $\Core$ of $D_u$ invariant and therefore the domain of $D_u$ is also stable under these right multiplication.

Using the decomposition of $D_u$ into Parts (0), (I) and (II), it appears that Part~(0) commutes with the right action of~$A^\text{op}$ and therefore does not contribute to the commutator. We treat Parts~(I) and~(II) separately. Keeping only Part (I) in the expression $D_u( \omegat \cdot a')$ for $\omegat = a \otimes v_1 \wedge \cdots \wedge v_k$, we get
\begin{gather*}
\sum_j \partial_j\big(a a' e^{-(n/2 - k)h u} \big) e^{(n/2 - (k+1))h u} \otimes T_j( v_1 \wedge \cdots \wedge v_k) \\
\qquad{}
= \sum_j \big(\partial_j( a) a' e^{-(n/2 - k)h u} + a \partial_j(a') e^{-(n/2 - k)h u}\big) e^{(n/2 - (k+1))h u}
 \otimes T_j( v_1 \wedge \ldots \wedge v_k) \\
 \qquad\quad{}
+ \sum_{j} a a' \partial _j\big(e^{-(n/2 - k)h u}\big) e^{(n/2 - (k+1))h u} \otimes T_j(v_1 \wedge \cdots \wedge v_k)\\
\qquad{}
= \sum_{j} \big(\partial _j(a) a' e^{- h u} + a \partial _j(a') e^{-hu} + a a' \partial _j\big(e^{-(n/2 - k)h u}\big) e^{(n/2 - (k+1))h u} \big)\\
\qquad\quad{}
 \otimes T_j(v_1 \wedge \cdots \wedge v_k).
\end{gather*}
We compare this expression to $D_u(\omegat) \beta(a')$, i.e.,
\begin{gather*}
\sum_j \partial_j\big(a e^{-(n/2 - k)h u} \big) e^{(n/2 - (k+1))h u} e^{h u} a' e^{-hu} \otimes T_j( v_1 \wedge \cdots \wedge v_k) \\
\qquad{}= \sum_j \partial_j(a) a' e^{-h u} \otimes T_j( v_1 \wedge \cdots \wedge v_k) \\
\qquad\quad{} + \sum_{j} a \partial _j\big(e^{-(n/2 - k)h u} \big) e^{(n/2 - (k+1))h u} e^{h u} a' e^{-hu} \otimes T_j( v_1 \wedge \cdots \wedge v_k).
\end{gather*}
In these two sums, the only terms that could lead to an unbounded contribution are those containing~$\partial _j(a)$, but these two terms cancel. At this point, we must perform the same computation on Part~(II) to make sure that the automorphism~$\beta$ is also suitable for this case. A~very similar computation proves that this it is indeed the case~-- the key property is that $e^{-(n/2 - k)h u} e^{(n/2 - (k+1))h u} = e^{- hu} = e^{(n/2 - k)h u} e^{-(n/2 - (k-1))h u}$~-- and thus the operator (def\/ined \textit{a~priori} only on~$\Core$)
\begin{gather*}
D_u \pi^\text{op}\big( (a')^\text{op} \big) - \pi^\text{op}( \beta(a')^\text{op}) D_u,
\end{gather*}
where $\pi^\text{op}\big( (a')^\text{op} \big) = \RightMult_{a'} \otimes \id_{\bigwedge^\bullet \LieG^*}$, extends to a bounded operator on~$\Hh$.
\end{proof}

Since the operator $D_u$ is odd with respect to the grading operator $\gamma$, we can write $D_u$ as combination of $D_u^+ \colon \Hh_\text{even} \to \Hh_\text{odd}$ and $D_u^- \colon \Hh_\text{odd} \to \Hh_\text{even}$. The odd Fredholm operator admits a (possibly) nontrivial index def\/ined as
\begin{gather}
\label{Eqn:DefOddIndex}
\ind_\text{odd}(D_u) = \dim \ker(D_u^+) - \dim \ker(D_u^-)
\end{gather}
(see, e.g., \cite[equation~(9.36), p.~397]{EltNCG}).

\section[Existence of a Chern--Gauss--Bonnet theorem for conformal perturbations of $C^*$-dynamical systems]{Existence of a Chern--Gauss--Bonnet theorem\\ for conformal perturbations of $\boldsymbol{C^*}$-dynamical systems}
\label{Sec:CGBTheorem}

In this section we show that the Hodge decomposition theorem proved in Section~\ref{Sec:HodgeOperator}
indicates the existence of an analog of the Chern--Gauss--Bonnet
theorem for the $C^*$-dynamical systems
studied in the present article. Let us explain the classical case
before stating the statement for our setting. Indeed, because of the natural
isomorphism between the space of harmonic dif\/ferential
forms and the de Rham cohomology groups, for a classical
closed manifold~$M$, the index of the operator
$d+d^*\colon \Omega^\text{even} M \to \Omega^\text{odd} M$ is
equal to the Euler characteristic of $M$. On the other
hand the McKean--Singer index theorem asserts that the index is
given by
\begin{gather*}
\ind\big(d+d^*\colon \Omega^\text{even} \to \Omega^\text{odd} \big)=\sum_{i=0}^{\dim M} (-1)^i\Tr \big(e^{-t \triangle_i}\big),
\end{gather*}
where $\triangle_i = d^* d + d d^*$ is the Laplacian on the
space of $i$-dif\/ferential forms on $M$, and $t$ is any positive number.
This formula, furthermore, contains local geometric information as~$t \to 0^+$,
since there is a small time asymptotic expansion of the form
\begin{gather*}
\Tr \big(e^{-t \triangle_i}\big)
\sim
t^{- \dim M/2} \sum_{j=0}^\infty a_{2j}(\triangle_i) t^j.
\end{gather*}
The coef\/f\/icients $a_{2j}(\triangle_i)$ are local geometric invariants,
which depend on the high frequency behaviour of the eigenvalues of the
Laplacian and are the integrals of some invariantly def\/ined local functions
$a_{2j}(x, \triangle_i)$ against the volume form of~$M$.
Independence of the index from~$t$ implies that the alternating sum of the
constant terms in the above asymptotic expansions for~$\triangle_i$ gives the
index. Hence, using the Hodge decomposition theorem,
\begin{gather*}
\chi(M)=\ind\big(d+d^*\colon \Omega^\text{even} \to \Omega^\text{odd} \big)
=
\int_M \sum_{i=0}^{\dim M} (-1)^ia_{\dim M}(x,\triangle_i) \, d \text{vol}_g.
\end{gather*}
In fact, the integrand in the latter coincides with the Pfaf\/f\/ian of the
curvature form, which is a~remarkable and dif\/f\/icult identif\/ication~\cite{HeatEquationABP}.

With notations and assumptions as in Section~\ref{Sec:HodgeOperator}, we obtain the following
result which indicates the existence of an analog of the Chern--Gauss--Bonnet theorem
in the setting of $C^*$-dynamical systems studied in this article.

\begin{Theorem}
The Euler characteristic $\chi$ of the complex $\Complex[u]$ is related to the odd index defined in~\eqref{Eqn:DefOddIndex}
\begin{gather*}
\chi
 =
\sum_{k=0}^n (-1)^k \dim H^k \Complex[u]
 =
\sum_{k=0}^n (-1)^k \ker(\Delta_k )
 =
\ind_\textnormal{odd} (D_u),
\end{gather*}
and is independent of the conformal factor $e^{-h}$.
\end{Theorem}

\begin{proof}
The f\/irst equality is actually the def\/inition of the Euler characteristic $\chi$. The second equality is an immediate consequence of Corollary~\ref{Cor:CohomGp}. The third equality and the last statement can be justif\/ied by using Remark~\ref{remarkfinitedimensionalityofcoh}
and Proposition~\ref{Prop:Stability}. That is, $\omegat \in \ker(\Delta_k)$ means in particular that~$\omegat$ is in the domain of~$\Delta_k$, which is included in the domain of~$D_u$ (by def\/inition). We then have
\begin{gather*}
0 = \langle \omegat, \Delta_k \omegat \rangle = \langle D_u \omegat, D_u \omegat \rangle,
\end{gather*}
which proves that $\omegat \in \ker(D_u)$. For a $k$-form $\omegat$, the converse is obvious. It follows that $\ker(D_u^+) = \bigoplus_{k \geqslant 0} \ker(\Delta_{2 k})$ and $\ker(D_u^-) = \bigoplus_{k \geqslant 0} \ker(\Delta_{2 k+1})$, which yields
\begin{gather*}
\ind_\text{odd}(D_u) = \dim \ker(D_u^+) - \dim \ker(D_u^-)
= \bigoplus_{k \geqslant 0} \dim \ker(\Delta_{2 k}) - \bigoplus_{k \geqslant 0} \ker(\Delta_{2 k+1}) = \chi.
\end{gather*}
The dimension of these groups are independent of the conformal factor $e^{-h}$ as a consequence of Proposition~\ref{Prop:Stability}.
\end{proof}

\begin{Remark}
An alternative proof of the index property using only bounded operators can be obtained using Sobolev spaces. For a clear account of these spaces and their analytic properties in our setting, we refer the reader to the paper~\cite{TrSpLieGpWahl}.
Also, in order to have a~complete analog of the Chern--Gauss--Bonnet theorem, one needs to f\/ind a local geometric formula
for the index, which is proved above to be a conformal invariant.
\end{Remark}

The heat kernels of Laplacians of conformally perturbed metrics on certain noncommutative spaces such as the noncommutative $n$-tori $\mathbb{T}_{\Theta}^n$ admit asymptotic expansions of the form
\begin{gather}
\label{Eqn:AsymptExp0}
\Tr( e^{-t \Delta_k}) \sim \sum_{j=0}^\infty a_j( \Delta_k) t^{(j-n)/2}, \qquad t \to 0^+.
\end{gather}
In fact, for noncommutative tori, each Laplacian $\Delta_k$ is an elliptic selfadjoint dif\/ferential operator of order 2, and
asymptotic expansions of this form can be derived by using
the heat kernel method explained in \cite{Gilkey} while employing Connes' pseudodif\/ferential calculus~\cite{CAlgGeoDiff}.
This method was indeed used in
\cite{ModCurvCM,GaussBonnet-ConnesTretkoff, GaussBonnet-FK, ScalarCurvNCtorusFK, ScalarCurv4NCTFK}, for
calculating and studying the term in the expansion that is related to the scalar curvature of noncommutative two and four tori.
Going through this process for noncommutative tori $\mathbb{T}_{\Theta}^n$, one can see that the odd coef\/f\/icients
in the latter asymptotic expansion will vanish, since in their explicit formula in terms of the pseudodif\/ferential symbol
of $\Delta_k$, there is an integration over the Euclidean space $\mathbb{R}^n$ of an odd function involved (see \cite[p.~54 and Theorem~1.7.6, p.~58]{Gilkey}). Thus in the case of the noncommutative torus $\mathbb{T}_{\Theta}^n$ we can write~\eqref{Eqn:AsymptExp0} as
\begin{gather}
\label{Eqn:AsymptExp}
\Tr\big( e^{-t \Delta_k}\big) \sim t^{-n/2} \sum_{j=0}^\infty a_{2j} ( \Delta_k) t^j, \qquad t \to 0^+.
\end{gather}
Now, using the McKean--Singer index formula \cite[Lemma~1.6.5, p.~47]{Gilkey} and our analog of Hodge decomposition theorem, for any $t > 0$ we have
\begin{gather}
\label{Eqn:Chi}
\chi = \sum_{k=0}^n (-1)^k \Tr\big( e^{-t \Delta_k}\big).
\end{gather}
Thus from equations~\eqref{Eqn:AsymptExp} and~\eqref{Eqn:Chi}, and using the independence
of the Euler characteristic from $t$ which implies that only the constant term from~\eqref{Eqn:AsymptExp} contributes
to the calculation of the Euler characteristic, we can write
\begin{gather*}
\chi = \sum_{k=0}^n (-1)^k a_{n}(\Delta_k) = \sum_{k=0}^n (-1)^k \varphi_0(\Rr_k),
\end{gather*}
where the local geometric invariants $\Rr_k$ are derived from the pseudodif\/ferential symbols of the Laplacian $\Delta_k$, by a heat kernel method. This method was used for example in \cite{ModCurvCM, ScalarCurvNCtorusFK, ScalarCurv4NCTFK} for computation of scalar curvature for noncommutative two
and four tori. The alternating sum of the $\Rr_k$ gives a noncommutative analog of the local expression for the Euler class.

\section{Summary and conclusions}
\label{Sec:Conclusions}

The Chern--Gauss--Bonnet theorem is an important generalization
of the Gauss--Bonnet theorem for surfaces, which states that
the Euler characteristic of an even-dimensional Riemannian
manifold can be computed as the integral of a characteristic
class, namely the Pfaf\/f\/ian of the curvature form, which is
a local invariant of the geometry. In particular, it shows that
the integral of this geometric invariant is independent of the
metric and depends only on the topology of the manifold.
The results obtained in this paper show that the analog
of this theorem holds for a general ergodic $C^*$-dynamical
system, whose algebra and Lie group are not necessarily
commutative. To be more precise, the family of metrics considered
for a dynamical system is obtained by using an invertible positive
element of the $C^*$-algebra to conformally perturb
a~f\/ixed metric def\/ined via the unique invariant trace, and our
result is about the
invariance of a~quantity, which is a~natural analog of the
Euler characteristic, from the conformal factor.

This type of results were previously proved for the
noncommutative two torus $\mathbb{T}_\theta^2$. That is, the analog of the
Gauss--Bonnet theorem was proved in \cite{GaussBonnet-ConnesTretkoff} and extended
to general translation invariant complex structures on these very
important but particular $C^*$-algebras in \cite{GaussBonnet-FK},
where a conformal factor varies the metric. The
dif\/ferential geometry of $C^*$-dynamical systems were
developed and studied in \cite{CAlgGeoDiff}, where the noncommutative
two torus $\mathbb{T}_\theta^2$ played a crucial role. However the investigation
of the analog of the Gauss--Bonnet theorem for $\mathbb{T}_\theta^2$, when the f\/lat metric is
conformally perturbed was pioneered in \cite{PreprintCC},
where after heavy calculations, some noncommutative
features seemingly indicated that the theorem does
not hold. However, stu\-dying the spectral
action in the presence of a dilaton \cite{SpActScaleChC}, the development of
the theory of twisted spectral triples \cite{TrSpTypeIII}, and further studies of
examples of complex structures on noncommutative manifolds \cite{LandietalComplex},
led to convincing observations that the Gauss--Bonnet theorem
holds for the noncommutative two torus. Then, by further analysis
of the expressions and functions of a~modular automorphism
obtained in~\cite{PreprintCC}, Connes and Tretkof\/f proved the desired
result in~\cite{GaussBonnet-ConnesTretkoff} for the simplest translation invariant conformal
structure, and the generalization of their result was established
in~\cite{GaussBonnet-FK} (where the use of a computer for the heavy computations
was inevitable).

It is remarkable that, a non-computational proof
of the Gauss--Bonnet theorem for the noncommutative two torus
is given in~\cite{ModCurvCM}, which is based on the work~\cite{BransonetalConInd}, where
the conformal index of a Riemannian manifold is def\/ined using properties of conformally covariant operators and the variational
properties of their spectral zeta functions. Therefore, since computations
are enormously more involved in dimensions higher than two, it is
of great importance to use spectral methods to show the existence
of the analog of the Chern--Gauss--Bonnet theorem, which is presented
in this article, not only for nocommutative tori, but for general $C^*$-dynamical
systems. We have also paid special attention to the spectral properties
of the analog of the Hodge--de~Rham operator of the perturbed metric: we have
proved its selfadjointness and shown that the spectral dimension is preserved.
We have then shown that this operator gives rise to a spectral triple with the unitary
left action of the algebra, and gives a twisted spectral triple with the unitary action of the
opposite algebra on the right, generalizing the construction in \cite{GaussBonnet-ConnesTretkoff} on the noncommutative
two torus and providing abstractly a large family of twisted spectral triples.

\subsection*{Acknowledgements}
The authors thank the Hausdorf\/f Research Institute
for Mathematics (HIM) for their hospitality and support during
the trimester program on Noncommutative Geometry
and its Applications in 2014, where the present work
was partially carried out. They also thank the anonymous referees
for their constructive feedback.
Parts of this article were obtained and written while the second author was working as a~postdoc at the University of Glasgow. He would like to thank C.~Voigt for enabling his stay in Scotland.
	
\pdfbookmark[1]{References}{ref}
\LastPageEnding

\end{document}